%% file: root.tex
\documentclass{amsart}

\usepackage{amssymb}
\input{macros}

\theoremstyle{plain}

\newtheorem{Theorem}{Theorem}[section]

\newtheorem{Lemma}[Theorem]{Lemma}
\newtheorem{Corollary}[Theorem]{Corollary}
\newtheorem{Hypothesis}[Theorem]{Hypothesis}

\newtheorem*{TheoremA}{Theorem~A}

\theoremstyle{definition}

\theoremstyle{remark}

\newtheorem*{UnnumberedRemark}{Remark}

\numberwithin{equation}{section}

\begin{document}

\title{A characterization of $A$-simple groups}


\author{Paul Flavell}
\address{The School of Mathematics\\University of Birmingham\\Birmingham B15 2TT\\Great Britain}
\email{P.J.Flavell@bham.ac.uk}
\thanks{A considerable portion of this research was done whilst the author was in receipt
of a Leverhulme Research Project grant and during visits to the Mathematisches Seminar,
Christian-Albrechts-Universit\"{a}t, Kiel, Germany.
The author expresses his thanks to the Leverhulme Trust for their support and
to the Mathematisches Seminar for its hospitality.}

\subjclass[2010]{Primary 20D45 20D05 20E34 }

\date{}

\begin{abstract}
    Let $A$ be an elementary abelian $r$-group with rank at least $3$
    that acts faithfully on the finite $r'$-group $G$.
    Assume that $G$ is $A$-simple,
    so that $G = K_{1} \times\cdots\times K_{n}$
    where $K_{1},\ldots,K_{n}$ is a collection of simple subgroups
    of $G$ that is permuted transitively by $A$.
    The purpose of this paper is to characterize $G$ and the collection
    of fixed point subgroups $\{ C_{G}(a) \;|\; a \in A^{\#} \}$.
    An application of this result will be a new proof of 
    McBride's Nonsolvable Signalizer Functor Theorem.
\end{abstract}

\maketitle
\input{intro}

\input{p}

\input{f}
\input{s}
\input{t}

\input{bib}

\end{document}

%% file: macros.tex
\newcommand{\cz}[2]{C_{#1}(#2)}

\newcommand{\n}[2]{N_{#1}(#2)}

\newcommand{\op}[2]{O_{#1}(#2)}

\newcommand{\ff}[1]{F(#1)}
\newcommand{\compp}[1]{\operatorname{comp}(#1)}
\newcommand{\comp}[2]{\operatorname{comp}_{#1}(#2)}

\newcommand{\layerr}[1]{E(#1)}

\newcommand{\gfitt}[1]{F^{*}(#1)}

\newcommand{\zz}[1]{Z(#1)}

\newcommand{\card}[1]{|\,#1\,|}

\newcommand{\syl}[2]{\operatorname{Syl}_{#1}(#2)}
\newcommand{\aut}[1]{\operatorname{Aut}(#1)}

\newcommand{\gen}[2]{\langle \;#1 \mid #2\; \rangle}
\newcommand{\listgen}[1]{\langle \;#1\; \rangle}

\newcommand{\rank}[1]{\operatorname{rank}(#1)}

\newcommand{\set}[2]{\{ \;#1 \mid #2\;\}}
\newcommand{\listset}[1]{\{ \,#1\, \}}

\newcommand{\sym}[1]{\operatorname{Sym}(#1)}
\newcommand{\isomorphic}{\cong}
\newcommand{\normal}{\,\unlhd\,}
\newcommand{\subnormal}{\,\unlhd\unlhd\,}

\newcommand{\openz}{{\mathbb Z}}

\newcommand{\br}[1]{\overline{#1}}
\newcommand{\nonid}{^\#}

\newcommand{\Amax}{\mathcal A_{\mbox{\it max}}}
\newcommand{\Amin}{\mathcal A_{\mbox{\it min}}}

%% file: intro.tex
\section{Introduction}\label{intro}
Throughout this paper,
$r$ is a prime and $A$ is an elementary abelian $r$-group
with rank at least $3$.
Suppose that $A$ acts faithfully on the finite $r'$-group $G$
and that $G$ is $A$-simple.
Then $G = K_{1} \times\cdots\times K_{n}$ where $K_{1},\ldots,K_{n}$
is a collection of simple subgroups of $G$ that is permuted
transitively by $A$.
The \emph{type} of $G$ is defined to be the isomorphism class of $K_{1}$.
Let $a \in A\nonid$ and suppose that $K_{1}$ is one of the known
simple groups.
Since $A$ is elementary abelian,
exactly one of the following holds:
\begin{enumerate}
    \item[(a)]  $\listgen{a}$ acts semiregularly on $\listgen{K_{1},\ldots,K_{n}}$
                and $\cz{G}{a}$ is $A$-simple with the same type as $G$.

    \item[(b)]  $a$ normalizes each $K_{i}$, $\layerr{\cz{K_{i}}{a}}$ is simple,
                $\ff{\cz{K_{i}}{a}} = 1$ and $\layerr{\cz{G}{a}}$ is $A$-simple.

    \item[(c)]  $a$ normalizes each $K_{i}$ and $\cz{G}{a}$ is solvable.
\end{enumerate}
Moreover in cases (a) and (b), \[
    \cz{A}{\layerr{\cz{G}{a}}} = \listgen{a}.
\]
These facts are established in \cite{I}.
For all $a,b \in A\nonid$ we have \[
    \layerr{\layerr{\cz{G}{a}} \cap \cz{G}{b} } \leq \layerr{\cz{G}{b}}
\]
since in (b),
$\cz{G}{a}/\layerr{\cz{G}{a}}$ is solvable.
In addition \[
    G = \gen{ \layerr{\cz{G}{a}} }{ a \in A\nonid }.
\]
The purpose of this paper is to prove that these conditions characterize $G$
and the collection of fixed point subgroups $\set{ \cz{G}{a} }{ a \in A\nonid }$.
We shall prove:
\begin{TheoremA} \label{thmA}
    Let $r$ be a prime and $A$ an elementary abelian $r$-group of rank at least $3$
    that acts on the group $G$.
    Suppose that for each $a \in A\nonid$ that $\theta(a)$ is a
    finite $A$-invariant $r'$-subgroup of $\cz{G}{a}$
    and that the following hold:
    \begin{enumerate}
        \item[(a)]  If $a \in A\nonid$ with $\layerr{\theta(a)} \not= 1$
                    then $\layerr{\theta(a)}$ is an $A$-simple $K$-group,
                    $\ff{\theta(a)} = 1$ and $\cz{A}{\layerr{\theta(a)}} = \listgen{a}$.

        \item[(b)]  For all $a,b \in A\nonid$, \[
                        \theta(a) \cap \cz{G}{b} \leq \theta(b)
                    \]
                    and \[
                        \layerr{\layerr{\theta(a)} \cap \cz{G}{b}} \leq \layerr{\theta(b)}.
                    \]

        \item[(c)]  $G = \gen{ \layerr{\theta(a)} }{ a \in A\nonid }$.
    \end{enumerate}
    Then $G$ is a finite $r'$-group,
    it is $A$-simple and a $K$-group.
    Moreover \[
        \cz{G}{a} = \theta(a)
    \]
    for all $a \in A\nonid$.
\end{TheoremA}

A number of remarks are in order.
\begin{itemize}
    \item   The first condition on (b) asserts that $\theta$ is an $A$-signalizer functor on $G$.
            The conclusion asserts that $\theta$ is complete.
            Theorem~\ref{thmA} is therefore a special case of the
            Nonsolvable Signalizer Functor Theorem.

    \item   The Nonsolvable Signalizer Functor Theorem was proved by
            McBride \cite{McB1, McB2} and plays a fundamental role
            in the proof of the Classification of the Finite Simple Groups,
            see \cite{GLS1} for example.
            This paper is one of a sequence of papers on automorphisms of $K$-groups,
            one of whose aims is to give a new proof of the
            Nonsolvable Signalizer Functor Theorem.
            In both his papers,
            McBride proves results similar to Theorem~\ref{thmA}.
            Hence this paper may be regarded as a revision of
            part of McBride's work.

    \item   Theorem~\ref{thmA} is analogous to the standard form problems
            that arise in the proof of the Classification of the Finite Simple Groups.

    \item   Recall that a $K$-group is a finite group all of whose simple sections
            are known simple groups.
            Since the main application of the Nonsolvable Signalizer Functor Theorem
            is to the Classification,
            the $K$-group assumption in (a) causes no difficulty.
            A crucial consequence is that if $a$ is an automorphism with order $r$
            of the simple $r'$-group $K$ then $\cz{K}{a}$
            has a unique minimal normal subgroup.
            It would be interesting to see if it possible to establish
            this without the $K$-group assumption.
\end{itemize}

%% file: p.tex
\section{Preliminaries}\label{p}
\begin{Lemma}\label{p:1}
    \begin{enumerate}
        \item[(a)]  Suppose that $K$ is a simple $K$-group and an $r'$-group.
                    Then the Sylow $r$-subgroups of $\aut{K}$ are cyclic.

        \item[(b)]  Suppose that $A$ is an elementary abelian $r$-group that acts
                    coprimely on the $K$-group $G$.
                    \begin{itemize}
                        \item[(i)]  If $G$ is $A$-simple then $\gfitt{\cz{G}{A}}$
                                    is the unique minimal normal subgroup of $\cz{G}{A}$.

                        \item[(ii)] If $K \in \comp{A}{G}$ then $\cz{G}{\cz{K}{A}} = \cz{G}{A}$.
                    \end{itemize}

        \item[(c)]  Suppose that $B$ is an elementary abelian $r$-group and that $K$
                    is a simple $K$-group and an $r'$-group.
                    Then there exists a prime $p$ such that whenever $J$ is $B$-simple
                    of type $K$ then \[
                        J = \listgen{ \cz{J}{B}, \zz{P} }
                    \]
                    for all $P \in \syl{p}{J;B}$.
    \end{enumerate}
\end{Lemma}
\begin{proof}
    (a) and (b) follows from \cite[Theorem~4.1]{I}.

    (c). TO DO. This reduces quickly to the case when $J$ is simple
    and $B$ has order $r$.
\end{proof}

\begin{Lemma}\label{p:2}
    Suppose that $A$ acts coprimely on the $K$-group $X$ and that $X$ is $A$-simple.
    Suppose also that $B \leq A$ is noncyclic.
    Then \[
        X = \gen{ \cz{X}{b} }{ \mbox{$b \in B\nonid$ and $\cz{X}{b}$ is $A$-simple %
                                with the same type as $X$} }.
    \]
\end{Lemma}
\begin{proof}
    If $\cz{B}{X} \not= 1$ then $X = \cz{X}{b}$ for any $b \in B\nonid$
    and there is nothing to prove.
    Hence we may assume that $\cz{B}{X} = 1$.
    Then we may replace $A$ by $A/\cz{A}{X}$ and assume that $\cz{A}{X} = 1$.
    Set $A_{\infty} = \ker(A \longrightarrow \sym{\compp{X}} )$.
    \cite[Lemma~6.4]{I} implies $A_{\infty}$ is cyclic so $B \cap A_{\infty} < B$.
    Let \[
        Y = \gen{ \cz{X}{b} }{ b \in B\setminus B\cap A_{\infty} }.
    \]
    \cite[Lemma~6.5]{I} implies  that $\cz{X}{b}$ is $A$-simple with the same type
    as $X$ for all $b \in B\setminus A_{\infty}$.
    Hence it suffices to show that $Y = X$.
    \cite[Lemma~6.5(d)]{I} implies $\cz{A}{\cz{X}{b}} = \listgen{b}$
    for all $b \in B\setminus B \cap A_{\infty}$.
    Then $\cz{A}{Y} \leq \listgen{b}$ so as $B$ is noncyclic
    and $B \cap A_{\infty}$ is cyclic,
    it follows that $\cz{A}{Y} = 1$.
    Note that $\cz{X}{b}$ is overdiagonal,
    which implies that $\cz{X}{b}$ projects onto each component of $X$
    for each $b \in B\setminus B\cap A_{\infty}$,
    so $Y$ is overdiagonal.
    \cite[Lemma~6.6]{I} implies $Y = \cz{X}{D}$ for some $D \leq A$.
    But $\cz{A}{Y} = 1$ so $D = 1$ and $Y = \cz{X}{1} = X$.
\end{proof}

\begin{Lemma} \label{p:3}
    Let $R$ be a group that acts coprimely on the group $G$.
    Assume that \[
        G = G_{1} *\cdots* G_{n}.
    \]
    where $\listset{G_{1},\ldots,G_{n}}$ is a collection of perfect subgroups
    of $G$ on which $R$ acts semiregularly.
    Then $\cz{G}{R}$ is perfect.
\end{Lemma}
\begin{proof}
    Note that $*$ denotes a central product.
    Let $G^{*} = G_{1} \times \cdots \times G_{n}$
    and observe that $R$ acts on $G^{*}$.
    The map $G^{*} \longrightarrow G$ defined by
    $(g_{1},\ldots,g_{n}) \mapsto g_{1}\cdots g_{n}$
    is an $R$-epimorphism.
    By Coprime Action,
    $\cz{G}{R}$ is the image of $\cz{G^{*}}{R}$.
    Visibly $\cz{G^{*}}{R}$ is perfect.
\end{proof}

\begin{Lemma}\label{p:4}
    Suppose $G$ is a perfect group and $G/\zz{G}$ is a finite $r'$-group.
    Then $G$ is a finite $r'$-group.
\end{Lemma}
\begin{proof}
    \cite[Theorem~9.8, p.250]{S} implies that $G$ is finite.
    Now $\op{r}{\zz{G}}$ is a normal Hall subgroup of $G$.
    Apply the Schur-Zassenhaus Theoren and the fact that $G$ is perfect.
\end{proof}

%% file: f.tex
\section{Part 1} \label{f}
Henceforth we assume the hypotheses of Theorem~\ref{thmA}.
For each $a \in A\nonid$ set \[
    E_{a} = \layerr{\theta(a)}.
\]
If $E_{a} \not= 1$ set \[
    A_{a} = \ker \left( A \longrightarrow \sym{\compp{E_{a}}} \right)
\]
otherwise set $A_{a} = A$.

Consider the types of $E_{a}$ as $a$ ranges over $A\nonid$
and choose a type $K$ of maximal order.
Note that $K$ exists because $G = \gen{ E_{a} }{ a \in A\nonid }$.
Set
\begin{align*}
        \Amax &=    \set{ a \in A\nonid }{ \mbox{$E_{a}$ has type $K$} } \\
        \Amin &=    A\nonid \setminus \Amax.
\end{align*}

\begin{Lemma}\label{f:1}
    Let $a,b \in A\nonid$.
    Then \[
        \layerr{\cz{E_{a}}{b}} = \layerr{\cz{E_{b}}{a}}.
    \]
\end{Lemma}
\begin{proof}
    By assumption,
    $\layerr{\cz{E_{a}}{b}} \leq E_{b}$ so \[
        \layerr{\cz{E_{a}}{b}} \leq \cz{E_{b}}{a} \leq \theta(b) \cap \cz{G}{a} %
        \leq \theta(b) \cap \theta(a).
    \]
    Now $E_{a} \normal \theta(a)$ so $\theta(a) \cap \theta(b)$
    normalizes $\layerr{\cz{E_{a}}{b}}$.
    Consequently $\layerr{\cz{E_{a}}{b}} \leq \layerr{\cz{E_{b}}{a}}$.
    The conclusion follows by symmetry.
\end{proof}

\begin{Lemma}\label{f:2}
    Let $a \in \Amax$.
    Then $\card{A_{a}} = r$ or $r^{2}$.
\end{Lemma}
\begin{proof}
    By hypothesis,
    $E_{a}$ is the direct product of simple groups
    that are permuted transitively by $A$.
    Since $A$ is elementary abelian,
    it follows that $A_{a}/\cz{A}{E_{a}}$ is isomorphic
    to a subgroup of the automorphism group of a component of $E_{a}$.
    Hence $A_{a}/\cz{A}{E_{a}}$ has order at most $r$ by Lemma~\ref{p:1}.
    Also by hypothesis,
    $\cz{A}{E_{a}} = \listgen{a}$,
    hence the result.
\end{proof}

\begin{Lemma}\label{f:3}
    Let $a \in \Amax$ and $b \in A\nonid$ with $\listgen{a} \not= \listgen{b}$.
    Exactly one of the following holds:
    \begin{enumerate}
        \item[(a)]  $b \not\in A_{a}$ and $a \not\in A_{b}$.
                    Moreover $b \in \Amax$ and $\card{A_{a}} = \card{A_{b}}$.

        \item[(b)]  $b \in A_{a}$ and $a \not\in A_{b}$.
                    Moreover $b \in \Amin$.

        \item[(c)]  $b \in A_{a}$ and $a \in A_{b}$.
    \end{enumerate}
\end{Lemma}
\begin{proof}
    We have \[
        E_{a} = K_{1} \times\cdots\times K_{\alpha}
    \]
    where $\listset{K_{1},\ldots,K_{\alpha}}$ are the components of $K$.
    Since $A$ is elementary abelian and transitive on $\listset{K_{1},\ldots,K_{\alpha}}$
    it follows that $A/A_{a}$ is regular.
    In particular, $\alpha = \card{A:A_{\alpha}}$.

    Suppose $b \not\in A_{a}$.
    Then $b$ acts as the product of $\alpha/r$ cycles of length $r$.
    Visibly
    \begin{equation} \tag{$1$}
        \layerr{\cz{E_{a}}{b}} = \cz{E_{a}}{b} \isomorphic K^{\alpha/r}.
    \end{equation}
    Suppose $b \in A_{a}$.
    Then $\cz{E_{a}}{b} = \cz{K_{1}}{b} \times\cdots\times \cz{K_{\alpha}}{b}$
    and $A$ is transitive on $\listset{\cz{K_{1}}{b},\ldots,\cz{K_{\alpha}}{b}}$.
    Then there is an action of $b$ on $K$ such that
    \begin{equation}\tag{$2$}
        \layerr{\cz{E_{a}}{b}} \isomorphic \layerr{\cz{K}{b}}^{\alpha}
    \end{equation}
    and $\cz{K}{b} \isomorphic \cz{K_{i}}{b}$ for all $i$.
    As $\listgen{b} \not= \listgen{a} = \cz{A}{E_{a}}$
    it follows that $\cz{K}{b} < K$.

    Lemma~\ref{f:2} implies that
    \begin{equation}\tag{$3$}
        \layerr{\cz{E_{a}}{b}} =  \layerr{\cz{E_{b}}{a}}.
    \end{equation}
    Suppose that $E_{b} = 1$.
    Then $\layerr{\cz{E_{a}}{b}} = 1$ so the previous paragraph
    implies $b \in A_{a}$.
    By the definition of $A_{b}$ we have $A = A_{b}$ so $a \in A_{b}$
    and (c) holds.
    Hence we may assume that $E_{b} \not= 1$.

    We have \[
        E_{b} = J_{1} \times\cdots\times J_{\beta}
    \]
    where $\listset{J_{1},\ldots,J_{\beta}}$ are the components
    of $E_{b}$ and $\beta = \card{A:A_{b}}$.
    Let $J$ be the type of $E_{b}$.
    Then there are expressions for $\layerr{\cz{E_{b}}{a}}$
    analogous to $(1)$ and $(2)$.

    Assume $b \not\in A_{a}$ and $a \not\in A_{b}$.
    Then $K^{\alpha/r} \isomorphic J^{\beta/r}$
    whence $K \isomorphic J$,
    so $b \in \Amax$.
    Also $\alpha = \beta$ so $\card{A_{a}} = \card{A_{b}}$ and (a) holds.

    Assume $b \in A_{a}$ and $a \not\in A_{b}$.
    Then $\layerr{\cz{K}{b}}^{\alpha} \isomorphic J^{\beta/r}$.
    Consequently $\card{J} \leq \card{\layerr{\cz{K}{b}}} < \card{K}$
    so $b \in \Amin$ and (b) holds.
    If $b \not\in A_{a}$ and $a \in A_{b}$,
    then similarly $\card{K} < \card{J}$,
    contrary to the choice of $K$,
    hence this case does not arise.
    The only remaining possibility is (c).
\end{proof}

\begin{Lemma}\label{f:4}
    Let $a \in \Amax$.
    \begin{enumerate}
        \item[(a)]  If $b \in A_{a} \cap \Amax$ then $A_{a} = A_{b}$.

        \item[(b)]  $\Amin \subseteq A_{a}$.

        \item[(c)]  If $b \in A\setminus A_{a}$ then $b \in \Amax$; $a \not\in A_{b}$;
                    $\cz{E_{a}}{b}$ and $\cz{E_{b}}{a}$ are $A$-simple;
                    $\cz{E_{a}}{b} = \cz{E_{b}}{a}$ and $\cz{E_{a}}{A} = \cz{E_{b}}{A}$.
    \end{enumerate}
\end{Lemma}
\begin{proof}
    (a). We may suppose $\listgen{a} \not= \listgen{b}$.
    Lemma~\ref{f:2} implies $A_{a} = \listgen{a,b}$.
    Now $b \in A_{a}$ so only case (c) of Lemma~\ref{f:3} may hold.
    Thus $a \in A_{b}$ and $A_{b} = \listgen{a,b} = A_{a}$.

    (b). Let $b \in \Amin$.
    Case (a) of Lemma~\ref{f:3} cannot hold so $b \in A_{a}$.

    (c). Since $b \not\in A_{a}$.
    Only Case (a) of Lemma~\ref{f:3} can hold so $b \in \Amax$ and $a \not\in A_{b}$.
    Since $b \not\in A_{a} = \ker \left( A \longrightarrow \sym{\compp{E_{a}}} \right)$
    it follows that $\cz{E_{a}}{b}$ is $A$-simple.
    Similarly so is $\cz{E_{b}}{a}$.
    Lemma~\ref{f:1} implies $\cz{E_{a}}{b} = \cz{E_{b}}{a}$.
    Now $a,b \in A$ whence $\cz{E_{a}}{A} = \cz{E_{b}}{A}$.
\end{proof}

\begin{Lemma}\label{f:5}
    Let $a \in \Amax$.
    Then \[
        A_{a} = \listgen{\Amin,a}.
    \]
    Moreover, $\card{\listgen{\Amin}} = 1$ or $r$.
\end{Lemma}
\begin{proof}
    Choose $b \in A \setminus A_{a}$.
    This is possible since $\card{A_{a}} = r$ or $r^{2}$
    by Lemma~\ref{f:2} and since $\rank{A} \geq 3$.
    Lemma~\ref{f:4}(c) implies $\cz{E_{a}}{b}$ is $A$-simple and
    \begin{equation}\tag{$1$}
        \cz{E_{a}}{b} = \cz{E_{b}}{a}.
    \end{equation}
    The components of $\cz{E_{a}}{b}$ correspond to the orbits
    of $\listgen{b}$ on $\compp{E_{a}}$.
    Since $A_{a}$ acts trivially on $\compp{E_{a}}$ it follows that
    $A_{a}$ acts trivially on $\compp{\cz{E_{a}}{b}}$.
    Then $(1)$ implies that $A_{a}$ acts trivially on
    the set of orbits of $\listgen{a}$ on $\compp{E_{b}}$.
    Now $A/A_{b}$ is abelian and regular on $\compp{E_{b}}$ so it
    follows that $\listgen{a}$ and $A_{a}$ induce the same group on $\compp{E_{b}}$.
    Consequently
    \begin{equation}\tag{$2$}
        A_{a} = \listgen{A_{a} \cap A_{b},a}.
    \end{equation}

    Lemma~\ref{f:4}(b) implies $\Amin \subseteq A_{a} \cap A_{b}$.
    Now $a \in A_{a}$ and $a \not\in A_{b}$ so $A_{a} \not= A_{b}$.
    If there exists $1 \not= b' \in (A_{a} \cap A_{b}) \setminus \Amin$
    then $b' \in A_{a} \cap A_{b} \cap \Amax$ so Lemma~\ref{f:4}(a)
    implies $A_{a} = A_{b'} = A_{b}$,
    a contradiction.
    We deduce that \[
        \listgen{\Amin} = A_{a} \cap A_{b} < A_{a}.
    \]
    Again, strict containment holding because $a \in A_{a} \setminus A_{b}$.
    Then $(2)$ implies $A_{a} = \listgen{\Amin,a}$.
    By Lemma~\ref{f:2},
    $\card{A_{a}} \leq r^{2}$ whence $\card{\listgen{\Amin}} \leq r$.
\end{proof}

Henceforth we set \[
    A_{\infty} = \listgen{ \Amin }.
\]
Then using Lemma~\ref{f:5}, we have:

\begin{Theorem}\label{f:6}
    Exactly one of the following holds:
    \begin{enumerate}
        \item[(a)]  $A_{\infty} = 1$.
                    For each $a \in A\nonid$, \[
                        \mbox{$E_{a} \isomorphic  K^{\card{A}/r}$ and $A_{a} = \listgen{a}$.}
                    \]

        \item[(b)]  $\card{A_{\infty}} = r$.
                    For each $a \in A \setminus A_{\infty}$, \[
                        \mbox{$E_{a} \isomorphic  K^{\card{A}/r^{2}}$ and $A_{a} = \listgen{A_{\infty},a}$.}
                    \]
    \end{enumerate}
\end{Theorem}

\begin{UnnumberedRemark}
    We will eventually show that $G$ is $A$-simple of type $K$.
    The two possibilities in Theorem~\ref{f:6} corresponds to
    whether $A$ acts faithfully on $\compp{G}$ or not;
    with $A_{\infty}$ being the kernel of this action.
    Moreover, we have established that the subgroups $\layerr{\theta(a)}$
    have the same structure as the corresponding subgroups
    in the target group.
\end{UnnumberedRemark}

%% file: s.tex
\section{Part 2}\label{s}
\begin{Lemma}\label{s:1}
    Suppose $B$ satisfies $A_{\infty} < B \leq A$.
    Then  \[
        G = \gen{ E_{b} }{ b \in B \setminus A_{\infty} }.
    \]
\end{Lemma}
\begin{proof}
    Let $H = \gen{ E_{b} }{ b \in B \setminus A_{\infty} }$.
    We claim that if $a \in A \setminus A_{\infty}$ then $E_{a} \leq H$.
    Indeed, Lemma~\ref{p:2} implies \[
        E_{a} = \gen{ \cz{E_{a}}{b} }%
            {\mbox{$b \in B\nonid$ and $\cz{E_{a}}{b}$ is $A$-simple %
            with the same type as $E_{a}$}}.
        \]
    Write $E_{a} = K_{1} \times\cdots\times K_{n}$ where $K_{1},\ldots,K_{n}$
    are the components of $E_{a}$.
    Suppose $e \in A_{\infty}\nonid$.
    Theorem~\ref{f:6} implies that $e$ normalizes each $K_{i}$ so
    $\cz{E_{a}}{e} = \cz{K_{1}}{e} \times\cdots\times \cz{K_{n}}{e}$.
    By assumption, $\cz{A}{E_{a}} = \listgen{a}$ so as $a \in A \setminus A_{\infty}$
    we have $\cz{K_{i}}{e} \not= K_{i}$.
    It follows that $\cz{E_{a}}{e}$ is not $A$-simple with the same type as $E_{a}$.
    Suppose $b \in B\nonid$ and $\cz{E_{a}}{b}$ is $A$-simple with the
    same type as $E_{a}$.
    Then $b \not\in A_{\infty}$.
    Lemma~\ref{f:1} implies $\cz{E_{a}}{b} \leq E_{b} \leq H$
    and the claim follows.

    Suppose $e \in A_{\infty}\nonid$.
    We claim that $E_{e} \leq H$.
    Theorem~\ref{f:6} implies $A_{\infty}$ is cyclic so as $\rank{A} \geq 3$
    we may choose $D \leq A$ with $D$ noncyclic and $D \cap A_{\infty} = 1$.
    Lemma~\ref{p:2} implies \[
        E_{e} = \gen{ \cz{E_{e}}{d} }{ \mbox{$d \in D\nonid$ and $\cz{E_{e}}{d}$ is $A$-simple.} }
    \]
    If $d \in D\nonid$ with $\cz{E_{e}}{d}$ being $A$-simple
    then as before $\cz{E_{e}}{d} \leq E_{d}$.
    Now $d \not\in A_{\infty}$ so $E_{d} \leq H$ by the previous paragraph.
    We have shown that $E_{a} \leq H$ for all $a \in A\nonid$.
    By assumption, $G = \gen{E_{a}}{a \in A\nonid}$,
    completing the proof.
\end{proof}

Set \[
    \br{G} = G/\zz{G}.
\]
For each $B \leq A$ define \[
    \Omega(B) = \bigcup_{b \in B \setminus A_{\infty}} \comp{B}{E_{b}}.
\]
Recall that if $B$ acts on the group $X$ then the $B$-components of $X$
are the groups generated by the orbits of $B$ on $\compp{X}$.
The set of $B$-components of $X$ is denoted by $\comp{B}{X}$.

\begin{Lemma}\label{s:2}
    Suppose that $B$ satisfies $A_{\infty} < B < A$.
    Suppose also that ``does not commute'' is an equivalence relation on $\Omega(B)$.
    Then:
    \begin{enumerate}
        \item[(a)]  $\br{G}$ is the direct product of $\card{A:B}$ perfect subgroups.
                    These subgroups are permuted transitively by $A$ and the
                    $A$-stabilizer of any of them is $B$.

        \item[(b)]  $G$ is a finite $r'$-group.

        \item[(c)]  Let $a \in A \setminus B$ then \[
                        \mbox{$\cz{G}{a} = \theta(a) = E_{a}$ and $\cz{\zz{G}}{a} = 1$}.
                    \]
    \end{enumerate}
\end{Lemma}
\begin{proof}
    Let $\sim$ be the equivalence relation ``does not commute'' on $\Omega(B)$
    and let $n = \card{A:B}$.
    Let $b \in B \setminus A_{\infty}$.
    Theorem~\ref{f:6} implies that $\compp{E_{b}}$ consists of $\card{A}/r\card{A_{\infty}}$
    copies of $K$ that are permuted regularly by $A/\listgen{b,A_{\infty}}$.
    Now $\listgen{b,A_{\infty}} \leq B$ so $B$ has $n$ orbits on $\compp{E_{b}}$
    and we have \[
        E_{b} = X_{1} \times\cdots\times X_{n}
    \]
    where $X_{1},\ldots,X_{n}$ are the $B$-components of $E_{b}$.
    Note that $A$ is transitive on $\listset{X_{1},\ldots,X_{n}}$ and for each $i$,
    $B = \n{A}{X_{i}}$.

    We claim that $\listset{X_{1},\ldots,X_{n}}$ is a set of representatives
    for the equivalence classes of $\sim$.
    If $i \not= j$ then $[X_{i},X_{j}] = 1$ so $X_{i}$ and $X_{j}$
    lie in different classes.
    Let $Y \in \Omega(B)$.
    Note that $\cz{Y}{b} \not= 1$ since $Y$ is not nilpotent.
    Then $1 \not= \cz{Y}{b} \leq \theta(b)$.
    Now $E_{b} \not= 1$ so by assumption $\ff{\theta(b)} = 1$
    and then $\cz{\theta(b)}{E_{b}} = 1$.
    Hence there exists $i$ with $[X_{i},\cz{Y}{b}] \not= 1$.
    In particular, $X_{i} \sim Y$.
    The claim is established.

    For each $i$ let $G_{i}$ be the group generated by the equivalence
    equivalence class of $X_{i}$.
    Since $B$-components are perfect,
    so is each $G_{i}$.
    The definition of $\sim$ implies $[G_{i},G_{j}] = 1$ for all $i \not= j$.
    Lemma~\ref{s:1} implies $G = G_{1} *\cdots * G_{n}$ and so
    \[
        \br{G} = \br{G_{1}} * \cdots * \br{G_{n}}.
    \]
    Then $A$ is transitive on $\listset{\br{G_{1}},\ldots,\br{G_{n}}}$
    and for each $i$,
    $B = \n{A}{\br{G_{i}}}$.
    In particular, (a) holds.

    For each $i$ let $\pi_{i}:\br{G} \longrightarrow \br{G_{i}}$ be the projection map.
    Let $a \in A \setminus B$.
    Without loss, $X_{1}^{a} = X_{2}, X_{2}^{a} = X_{3}, \ldots$.
    In particular, $a$ is nontrivial on $\compp{E_{b}}$
    and so $\cz{E_{b}}{a}$ is $A$-simple.
    Lemma~\ref{f:1} implies $\cz{E_{b}}{a} \leq E_{a}$.
    Let $x_{1} \in X_{1}$,
    put $x_{2} = x_{1}^{a}, x_{3} = x_{2}^{a}, \ldots$
    and $x = x_{1}\cdots x_{r} \in \cz{E_{b}}{a} \leq E_{a}$.
    Then $\br{x} \in \br{E_{a}}$ and
    $\br{x_{1}} = \br{x}\pi_{1} \in \br{E_{a}}\pi_{1}$.
    We deduce that \[
        \br{X_{1}} \leq \br{E_{a}}\pi_{1}.
    \]
    Now suppose $Y_{1} \in \Omega(B)$ and $Y_{1} \sim X_{1}$.
    Then for some $b' \in B \setminus A_{\infty}$
    we have $Y_{1} \in \comp{B}{E_{b'}}$.
    We may write $E_{b'} = Y_{1} \times\cdots\times Y_{n}$
    with $Y_{i} \sim X_{i}$ for all $i$.
    Then $G_{i}$ is the group generated by the equivalence class of $Y_{i}$.
    The previous paragraph,
    with $(b',Y_{1})$ in the role of $(b,X_{1})$
    implies $\br{Y_{1}} \leq \br{E_{a}}\pi_{1}$.
    It follows that \[
        \br{G_{1}} \leq \br{E_{a}}\pi_{1} \leq \br{\theta(a)}\pi_{1} %
        \leq \br{\cz{G}{a}}\pi_{1} \leq \br{G_{1}}
    \]
    and then we have equality.

    Observe that $\br{G_{1}}$ is a finite $r'$-group since it is the image
    of the finite $r'$-group $\theta(a)$.
    Then $\br{G}$ is a finite $r'$-group.
    Now $G$ is perfect since it is generated by the perfect subgroups $G_{i}$.
    Lemma~\ref{p:4} implies that $G$ is a finite $r'$-group so (b) holds.

    Since $\listgen{a}$ is semiregular on $\listset{\br{G_{1}},\ldots,\br{G_{n}}}$
    we have $\br{\cz{G}{a}} \cap \ker \pi_{1} = 1$ so
    the previously proved equality forces \[
        \br{E_{a}} = \br{\theta(a)} = \br{\cz{G}{a}}.
    \]
    In particular, $E_{a}(\zz{G} \cap \cz{G}{a}) = \cz{G}{a}$.
    Lemma~\ref{p:3} implies that $\cz{G}{a}$ is perfect.
    Consequently $E_{a} = \cz{G}{a}$.
    Since $E_{a} \leq \theta(a)$ we obtain $E_{a} = \theta(a) = \cz{G}{a}$.
    Also as $E_{a} \not= 1$ we have $\ff{\theta(a)} = 1$
    so $\cz{\zz{G}}{a} = 1$ and (c) holds.
\end{proof}

\begin{Theorem}\label{s:3}
    Suppose $B_{1},\ldots,B_{n}$ satisfy:
    \begin{itemize}
        \item   For each $i$, $A_{\infty} < B_{i} < A$ and
                ``does not commute'' is an equivalence relation on $\Omega(B_{i})$.

        \item   $A_{\infty} = B_{1} \cap \cdots \cap B_{m}$.
    \end{itemize}
    Then $G$ is an $A$-simple $r'$-group.
    It is the direct product of $\card{A:A_{\infty}}$ copies of $K$,
    $A_{\infty} = \ker \left( A \longrightarrow \sym{\compp{G}} \right)$
    and for all $a \in A\nonid$, \[
        \theta(a) = \cz{G}{a}.
    \]
    In particular, $\theta$ is complete.
\end{Theorem}
\begin{proof}
    Let $a \in A \setminus A_{\infty}$.
    Then there exists $i$ with $a \not\in B_{i}$ so Lemma~\ref{s:2} implies \[
        \mbox{$\theta(a) = \cz{G}{a}$ and $\cz{\zz{G}}{a} = 1$}.
    \]
    Choose $D < A$ with $D \cap A_{\infty} = $ and $\rank{A} = 2$.
    Suppose $1 \not= e \in A_{\infty}$.
    Now $G$ is an $r'$-group by Lemma~\ref{s:2}
    so by considering the action of $D$ on $\cz{G}{e}$ we have
    \begin{align*}
        \cz{G}{e} &= \gen{\cz{G}{e} \cap \cz{G}{d} }{d \in D\nonid} \\
                  &= \gen{\cz{G}{e} \cap \theta(d) }{d \in D\nonid} \leq \theta(e).
    \end{align*}
    Hence \[
        \theta(e) = \cz{G}{e}.
    \]
    Consequently $\theta$ is complete.
    Moreover $\cz{\zz{G}}{d} = 1$ for all $d \in D\nonid$
    so as $D$ is noncyclic,
    this forces $\zz{G} = 1$.

    Let \[
        G = K_{1} \times\cdots\times K_{l}
    \]
    be a decomposition of $G$ into a direct product of nontrivial indecomposable groups.
    Since $G$ is perfect,
    so is each $K_{i}$ and this decomposition is unique by the Krull-Schmidt Theorem.
    In particular,
    $A$ permutes $\listset{K_{1},\ldots,K_{l}}$.

    Recall that $a \in A\setminus A_{\infty}$.
    Choose $i$ with $a \not\in B_{i}$.
    Now $\zz{G} = 1$ so Lemma~\ref{s:2} implies there exists a decomposition \[
        G = G_{1} \times\cdots\times G_{\card{A:B_{i}}}
    \]
    and that $\listgen{a}$ acts semiregularly on the factors.
    Let $1 \leq i \leq l$.
    Writing each $G_{j}$ as a direct product of indecomposable groups,
    we see that there exists $k$ with $K_{i} \leq G_{k}$.
    Now $a$ does not normalize $G_{k}$ so it does not normalize $K_{i}$.
    Consequently $\listgen{a}$ acts semiregularly on $\listset{K_{1},\ldots,K_{l}}$.
    In particular,
    $\cz{G}{a}$ is a direct product of groups,
    one for each orbit of $\listgen{a}$ on $\listset{K_{1},\ldots,K_{l}}$.
    Now $\cz{G}{a} = \theta(a) = E_{a}$ and $E_{a}$ is $A$-simple of type $K$.
    If $1 \not= e \in A_{\infty}$ then since $E_{e}$ is not $A$-simple of type $K$,
    we see that $e$ normalizes each $K_{i}$.
    Consequently \[
        A_{\infty} = \ker \left( A \longrightarrow \sym{\compp{G}} \right),
    \]
    completing the proof.
\end{proof}

%% file: t.tex
\section{Part 3}\label{t}
In order to complete the proof of Theorem~\ref{thmA},
by the previous theorem it suffices to show that if
$A_{\infty} < B < A$ with $\card{A:B} = r$ then
``does not commute'' is an equivalence relation
on $\Omega(B)$.
If $\rank{A} \geq 4$,
then $\rank{B} \geq 3$ so as $\theta$ is also a $B$-signalizer functor,
this follows from \cite[The Global Theorem, Theorem~10.2]{I}.
However, more work is needed to establish this fact
if $\rank{A} = 3$.

\begin{Hypothesis}\label{t:1}\mbox{}
    \begin{itemize}
        \item   $A_{\infty} < B < A$ with $\rank{B} = 2$.

        \item   $a,b \in B \setminus A_{\infty}$ with $B = \listgen{a,b}$.

        \item   $E_{a} = K_{1} \times\cdots\times K_{\alpha}$
                and $E_{b} = J_{1} \times\cdots\times J_{\alpha}$
                where $\alpha = \card{A}/r^{2}$ and
                $K_{1},\ldots,K_{\alpha}$ and $J_{1},\ldots,J_{\alpha}$
                are the $B$-components of $E_{a}$ and $E_{b}$ respectively.
    \end{itemize}
\end{Hypothesis}
\noindent To justify the assertion made in the third part
consider first the case that $A_{\infty} = 1$.
Theorem~\ref{f:6} implies that $E_{a}$ is the direct product
of $\card{A}/r$ copies of $K$ and $A_{a} = \listgen{a}$.
Since $a \in B$ and $\card{B} = r^{2}$ it follows that $B$ has $\alpha$
orbits on $\compp{E_{a}}$.
Consider the case $\card{A_{\infty}} = r$.
Theorem~\ref{f:6} implies that $E_{a}$ is the direct product
of $\card{A}/r^{2}$ copies of $K$ and $A_{a} = \listgen{A_{\infty},a}$.
Then $B = A_{a}$ so $B$ is trivial on $\compp{E_{a}}$.
Again, $B$ has $\alpha$ orbits on $\compp{E_{a}}$.

\begin{Lemma}\label{t:2}
    Assume Hypothesis~\ref{t:1}.
    After a suitable renumbering,
    the following hold:
    \begin{enumerate}
        \item[(a)]  $\gfitt{\cz{K_{i}}{B}} = \gfitt{\cz{J_{i}}{B}}$ for all $i$.

        \item[(b)]  $[\cz{K_{i}}{B},J_{j}] = 1 = [K_{i},\cz{J_{j}}{B}]$
                    for all $i \not= j$.
    \end{enumerate}
\end{Lemma}
\begin{proof}
    Let $N = \theta(B)$,
    so that $N = \cap_{c \in B\nonid}\theta(c)$.
    For any $c \in B\nonid$ we have $\cz{E_{a}}{B} \leq \theta(a) \cap \cz{G}{c} \leq \theta(c)$
    whence $\cz{E_{a}}{B} \leq N$.
    As $a \in B\nonid$ and $E_{a} = \layerr{\theta(a)}$ it follows that
    $AN$ normalizes $E_{a}$.
    In fact,
    as $AN$ centralizes $B$,
    it follows that $AN$ permutes the $B$-components of $E_{a}$.
    Now $E_{a}$ is $A$-simple so $A$ acts transitively.
    Moreover, $N \normal AN$ and $N$ is an $r'$-group.
    This implies that $N$ acts trivially.
    Consequently \[
        \cz{E_{a}}{B} = \cz{K_{1}}{B} \times\cdots\times \cz{K_{\alpha}}{B} %
        \quad\mbox{and}\quad \cz{K_{i}}{B} \normal N.
    \]
    Similarly \[
        \cz{E_{b}}{B} = \cz{J_{1}}{B} \times\cdots\times \cz{J_{\alpha}}{B} %
        \quad\mbox{and}\quad \cz{J_{i}}{B} \normal N.
    \]
    We have also shown that $\cz{K_{i}}{B}$ normalizes $J_{j}$
    and $\cz{J_{j}}{B}$ normalizes $K_{i}$ for all $i,j$.

    Using Lemma~\ref{p:1}(b)(ii) we have
    $\cz{N}{\cz{E_{b}}{B}} \leq \cz{\theta(b)}{\cz{E_{b}}{B}} \leq \cz{\theta(b)}{E_{b}}$.
    Now $E_{b} = \layerr{\theta(b)}$ so by assumption,
    $\ff{\theta(b)} = 1$ and hence $\cz{\theta(b)}{E_{b}} = 1$.
    Consequently \[
        \cz{N}{\cz{E_{b}}{B}} = 1.
    \]
    Let $1 \leq i \leq \alpha$.
    By the above,
    there exists $j$ such that $[\cz{K_{i}}{B},\cz{J_{j}}{B}] \not= 1$.
    Then \[
        1 \not= [\cz{K_{i}}{B},\cz{J_{j}}{B}] \leq \cz{K_{i}}{B} \cap \cz{J_{j}}{B} %
        \normal \listgen{\cz{K_{i}}{B}, \cz{J_{j}}{B}}.
    \]
    Lemma~\ref{p:1}(b)(i) implies that $\gfitt{\cz{K_{i}}{B}}$ is the
    unique minimal normal subgroup of $\cz{K_{i}}{B}$.
    Then $\gfitt{\cz{K_{i}}{B}} \leq \gfitt{\cz{K_{i}}{B} \cap \cz{J_{j}}{B}} %
    \leq \gfitt{\cz{J_{j}}{B}}$.
    By symmetry we have \[
        \gfitt{\cz{K_{i}}{B}} = \gfitt{\cz{J_{j}}{B}}.
    \]
    If $k \not= j$ then $\gfitt{\cz{J_{j}}{B}} \not= \gfitt{\cz{J_{k}}{B}}$
    because $J_{j} \cap J_{k} = 1$.
    It follows that $\gfitt{\cz{J_{k}}{B}} \not= \gfitt{\cz{K_{i}}{B}}$,
    so $j$ is uniquely determined.
    Hence $[\cz{K_{i}}{B}, \cz{J_{k}}{B}] = 1$ for all $k \not= j$.

    After a suitable renumbering,
    we have \[
        \gfitt{\cz{K_{i}}{B}} = \gfitt{\cz{J_{i}}{B}}
    \]
    for all $i$ and \[
        [\cz{K_{i}}{B}, \cz{J_{j}}{B}] = 1
    \]
    for all $i \not= j$.
    To prove (b) suppose that $i \not= j$.
    At the beginning of the proof,
    we observed that $\cz{K_{i}}{B}$ normalizes $J_{j}$.
    Then (b) follows from Lemma~\ref{p:1}(b)(ii).
\end{proof}

Henceforth, whenever we assume Hypothesis~\ref{t:1},
we assume the ordering has been chosen on accordance with Lemma~\ref{t:2}.

\begin{Lemma}\label{t:3}
    Assume Hypothesis~\ref{t:1}.
    Then \[
        \cz{K_{i}}{B} = \cz{J_{i}}{B}
    \]
    for all $i$.
\end{Lemma}
\begin{proof}
    Without loss $i=1$.
    Choose $c \in A \setminus B$.
    Now $A/A_{\infty}$ is regular on $\compp{E_{b}}$ and
    $J_{1},\ldots,J_{\alpha}$ are the subgroups generated by the orbits
    of $B$ on $\compp{E_{b}}$.
    It follows that $\listgen{c}$ is semiregular on $\listset{J_{1},\ldots,J_{\alpha}}$.
    Then visibly $\cz{E_{b}}{B} \cap \cz{G}{c}$
    induces the same group on $\cz{J_{1}}{B}$ as does $\cz{J_{1}}{B}$.
    Also, $\listgen{c}$ is semiregular on $\listset{K_{1},\ldots,K_{\alpha}}$
    so using Lemma~\ref{t:2} we see that $\cz{E_{a}}{B} \cap \cz{G}{c}$
    induces the same group on $\cz{J_{1}}{B}$ as does $\cz{K_{1}}{B}$.
    Using Lemma~\ref{f:4}(c), with $(b,c)$ in the role of $(a,b)$,
    we have \[
        \cz{E_{b}}{B} \cap \cz{G}{c} = \cz{G}{B} \cap \cz{E_{b}}{c} %
        = \cz{G}{B} \cap \cz{E_{c}}{b} = \cz{E_{c}}{B}.
    \]
    Similarly, $\cz{E_{a}}{B} \cap \cz{G}{c} = \cz{E_{c}}{B}$.
    We deduce that $\cz{K_{1}}{B}$ and $\cz{J_{1}}{B}$
    induce the same group on $\cz{J_{1}}{B}$.

    Let $X = \cz{\theta(B)}{\cz{J_{2}}{B} \times\cdots\times \cz{J_{\alpha}}{B}}$.
    Lemma~\ref{p:1}(b)(ii) implies that $[X, J_{2} \times\cdots\times J_{\alpha}] = 1$.
    Moreover, $\ff{\theta(b)} = 1$ so
    $\gfitt{\theta(b)} = J_{1} \times\cdots\times J_{\alpha}$
    and it follows that $X$ is faithful on $\cz{J_{1}}{B}$.
    Clearly $\cz{J_{1}}{B} \leq X$.
    By Lemma~\ref{t:2}, $\cz{K_{1}}{B} \leq X$.
    Since $\cz{J_{1}}{B}$ and $\cz{K_{1}}{B}$ induce the same group on $\cz{J_{1}}{B}$,
    we deduce that $\cz{J_{1}}{B} = \cz{K_{1}}{B}$.
\end{proof}

\begin{Corollary}\label{t:4}
    Let $a,b \in A \setminus A_{\infty}$.
    Then \[
        \cz{E_{a}}{b} = \cz{E_{b}}{a}.
    \]
\end{Corollary}
\begin{proof}
    Suppose $\listgen{a} = \listgen{b}$.
    Then $\theta(a) = \theta(a) \cap \cz{G}{b} \leq \theta(b)$.
    By symmetry, $\theta(a) = \theta(b)$ whence $E_{a} = E_{b}$
    and the conclusion follows.
    If $b \not\in A_{a}$ then the conclusion follows from Lemma~\ref{f:4}(c).
    Hence we may suppose that $b \in A_{a}$.
    Set $B = A_{a}$.
    Lemma~\ref{f:5} implies $B = \listgen{a,b}$.
    Also, $A_{\infty} < A_{a}$.
    Using Lemma~\ref{t:3} we have
    $\cz{E_{a}}{b} = \cz{E_{a}}{B} = \cz{E_{b}}{B} = \cz{E_{b}}{a}$.
\end{proof}

Recall that by hypothesis,
$\theta$ is an $A$-signalizer functor on $G$.
Also if $p$ is a prime then a $(p,\theta)$-subgroup of $G$
is an $A$-invariant $p$-subgroup $P$ with the property \[
    \cz{P}{a} \leq \theta(a)
\]
for all $a \in A\nonid$.
The set of maximal $(p,\theta)$-subgroups is denoted by $\syl{p}{G;\theta}$.
The Transitivity Theorem, \cite[p. 309]{KS} asserts that $\theta(A)$
acts transitively by conjugation on $\syl{p}{G;A}$.
Consequently $\cz{P}{a} \in \syl{p}{\theta(a)}$
for all $P \in \syl{p}{G;\theta}$ and $a \in A\nonid$.

\begin{Lemma}\label{t:5}
    Assume Hypothesis~\ref{t:1}.
    Let $p$ be a prime and $P \in \syl{p}{G;\theta}$.
    Then \[
        [\zz{P \cap E_{a}}, \zz{P \cap E_{b}}] = 1.
    \]
\end{Lemma}
\begin{proof}
    For each $c \in A \setminus A_{\infty}$ set $P_{c} = P \cap E_{c}$.
    Now $E_{c} \normal \theta(c)$ so the preceding discussions imply
    $P_{c} \in \syl{p}{E_{c};A}$.

    Let $c \in A \setminus B$.
    We claim that \[
        \cz{\zz{P_{a}}}{c} = \cz{\zz{P_{c}}}{a}.
    \]
    Note that $P_{a}$ is the direct product of its intersections
    with the components of $E_{a}$.
    Since $c \in A \setminus B$ and $A_{a} = \listgen{A_{\infty},a} \leq B$
    we have $c \not\in A_{a}$ so $\listgen{c}$ is semiregular
    on $\compp{E_{a}}$ and hence on the factors of $P_{a}$.
    It follows that $\zz{\cz{P_{a}}{c}} \leq \zz{P_{a}}$.
    Consequently \[
        \zz{\cz{P_{a}}{c}} = \cz{\zz{P_{a}}}{c}.
    \]
    Since $c \not\in A_{a}$,
    Lemma~\ref{f:4}(c) implies
    $a \not\in A_{c}$ and $\cz{P_{a}}{c} = \cz{P_{c}}{a}$.
    Then $\zz{\cz{P_{a}}{c}} = \zz{\cz{P_{c}}{a}}$.
    Repeating the above argument,
    we have $\zz{\cz{P_{c}}{a}} = \cz{\zz{P_{c}}}{a}$
    and the claim follows.
    For any such $c$ we also have $\cz{\zz{P_{b}}}{c} = \cz{\zz{P_{c}}}{b}$.
    We deduce that \[
        [\cz{\zz{P_{a}}}{c}, \cz{\zz{P_{b}}}{c}] = 1
    \]
    for all $c \in A \setminus B$.

    Assume now the lemma to be false.
    Then $[\zz{P_{a}},\zz{P_{b}}] \not= 1$.
    By Coprime Action,
    there exists $C < A$ with $\card{A:C} = r$ and
    $[\cz{\zz{P_{a}}}{C},\zz{P_{b}}] \not= 1$.
    Now $\rank{A} \geq 3$ so $C$ is noncyclic.
    Again by Coprime Action,
    there exists $c \in C\nonid$ with
    $[\cz{\zz{P_{a}}}{C}, \cz{\zz{P_{b}}}{c}] \not= 1$.
    In particular
    \begin{equation}\tag{$*$}
        [\cz{\zz{P_{a}}}{c},\cz{\zz{P_{b}}}{c}] \not= 1.
    \end{equation}
    The previous paragraph implies $c \in B$.
    If $B = \listgen{a,c}$ then as $b \in B$ and using Corollary~\ref{t:4}
    we have \[
        \cz{\zz{P_{a}}}{c} \leq P \cap \cz{E_{a}}{b} = P \cap \cz{E_{b}}{a} \leq P_{b},
    \]
    contrary to $(*)$.
    Thus $c \in \listgen{a}$.
    Similarly $c \in \listgen{b}$.
    But by Hypothesis~\ref{t:1},
    $\listgen{a} \not= \listgen{b}$,
    a contradiction.
\end{proof}

\begin{Lemma}\label{t:6}
    Assume Hypothesis~\ref{t:1}
    Then \[
        [K_{i},J_{j}] = 1
    \]
    for all $i \not= j$.
\end{Lemma}
\begin{proof}
    By Theorem~\ref{f:6},
    each member of $\Omega(B)$ is $B$-simple of type $K$
    so Lemma~\ref{p:1}(c) implies there exists a prime $p$
    such that \[
        J = \listgen{ \cz{J}{B}, \zz{P_{0}} }
    \]
    whenever $J \in \Omega(B)$ and $P_{0} \in \syl{p}{J;B}$.
    Choose $P \in \syl{p}{G;\theta}$.
    By the Transitivity Theorem,
    $\cz{p}{a} \in \syl{p}{\theta(a);A}$ so as $K_{i} \subnormal \theta(a)$
    we have $P \cap K_{i} \in \syl{p}{K_{i};B}$.
    Similarly, $P \cap J_{j} \in \syl{p}{J_{j};B}$.
    Consequently \[
        K_{i} = \listgen{ \cz{K_{i}}{B}, \zz{P \cap K_{i}} } \quad\mbox{and}\quad %
        J_{j} = \listgen{ \cz{J_{j}}{B}, \zz{P \cap J_{j}} }.
    \]
    Lemma~\ref{t:2}(b) implies $[\cz{K_{i}}{B},J_{j}] = 1 = [K_{i},\cz{J_{j}}{B}]$.
    As $E_{a} = K_{1} \times\cdots\times K_{\alpha}$
    we have $\zz{P \cap K_{i}} \leq \zz{P \cap E_{a}}$.
    Similarly $\zz{P \cap J_{j}} \leq \zz{P \cap E_{b}}$.
    Lemma~\ref{t:5} implies that $[\zz{P \cap K_{i}}, \zz{P \cap J_{j}}] = 1$.
    It follows that $[K_{i},J_{j}] = 1$.
\end{proof}

\begin{proof}[Proof of Theorem~\ref{thmA}]
    By Theorem~\ref{s:3} it suffices to choose $B \isomorphic \openz_{r} \times \openz_{r}$
    with $A_{\infty} < B < A$ and prove that
    ``does not commute'' is an equivalence relation on $\Omega(B)$.
    This will follow if we can show that \[
        [X,Y] \not= 1 \quad\mbox{if and only if}\quad \cz{X}{B} = \cz{Y}{B}
    \]
    whenever $X,Y \in \Omega(B)$.

    Choose $X,Y \in \Omega(B)$.
    By the definition of $\Omega(B)$,
    there exists $a,b \in B \setminus A_{\infty}$
    such that $X \in \comp{B}{E_{a}}$ and $Y \in \comp{B}{E_{b}}$.
    The ``if'' direction is clear since as $X$ is $B$-simple,
    $\cz{X}{B} \not= 1$.
    Thus to complete the proof,
    we assume $[X,Y] \not= 1$ and aim to show $\cz{X}{B} = \cz{Y}{B}$.

    If $\listgen{a} = \listgen{b}$ then since $\theta$ is an $A$-signalizer functor
    we have $\theta(a) = \theta(b)$,
    whence $E_{a} = E_{b}$.
    As distinct $B$-components of $E_{a}$ commute,
    the conclusion is clear.
    Hence we assume that $\listgen{a} \not= \listgen{b}$.
    Then $B = \listgen{a,b}$ and Hypothesis~\ref{t:1} is satisfied.
    In particular $X = K_{i}$ and $Y = J_{j}$ for some $i,j$.
    Lemmas~\ref{t:6} and \ref{t:3} imply $\cz{X}{B} = \cz{Y}{B}$.
    The proof is complete.
\end{proof}

%% file: bib.tex
\bibliographystyle{amsplain}